\documentclass[a4paper,reqno,11pt]{amsart}

\usepackage{color,amsmath,amssymb,hyperref,setspace,mathtools}
\usepackage[margin=2.7cm]{geometry}

\theoremstyle{plain}
\newtheorem{theorem}{Theorem}[section]
\newtheorem*{theorem*}{Theorem}
\newtheorem{lemma}[theorem]{Lemma}
\newtheorem{proposition}[theorem]{Proposition}
\newtheorem{corollary}[theorem]{Corollary}
\newtheorem*{corollary*}{Corollary}

\theoremstyle{remark}
\newtheorem{remark}[theorem]{Remark}

\theoremstyle{definition}
\newtheorem{definition}[theorem]{Definition}

\newcommand{\Q}{\mathbb{Q}}
\newcommand{\Z}{\mathbb{Z}}
\newcommand{\R}{\mathbb{R}}
\newcommand{\A}{\mathcal{A}}
\newcommand{\Abar}{\overline{\mathcal{A}}}
\newcommand{\C}{\mathcal{C}}
\newcommand{\I}{^{-1}}
\newcommand{\defeq}{\mathrel{\vcentcolon =}}

\DeclareMathOperator{\SL}{SL}
\DeclareMathOperator{\GL}{GL}

\DeclareMathOperator{\Br}{Br}
\DeclareMathOperator{\Gal}{Gal}
\DeclareMathOperator{\inv}{inv}
\DeclareMathOperator{\Stab}{Stab}
\DeclareMathOperator{\Fix}{Fix}
\DeclareMathOperator{\diag}{diag}

\numberwithin{equation}{section}

\begin{document}

\title{Division algebras and transitivity of group actions on buildings}

\author{Matthew C. B. Zaremsky}
\address{Department of Mathematics \\
Bielefeld University \\
33615 Bielefeld, Germany}
\email{zaremsky@math.uni-bielefeld.de}
\date{\today}
\thanks{The author is supported in part by the SFB~$701$ of the DFG}
\subjclass[2010]{Primary 20E42; Secondary 16K20, 21E24}
\keywords{Division algebra, building, strongly transitive, Weyl transitive}

\begin{abstract}
Let~$D$ be a division algebra with center~$F$ and degree~$d>2$. Let~$K|F$ be any splitting field. We analyze the action of~$D^\times$ and $\SL_1(D)$ on the spherical and affine buildings that may be associated to $\GL_d(K)$ and $\SL_d(K)$, and in particular show it is never strongly transitive. In the affine case we find examples where the action is nonetheless Weyl transitive. This extends results of Abramenko and Brown concerning the~$d=2$ case, where strong transitivity is in fact possible. Our approach produces some explicit constructions, and we find that for~$d>2$ the failure of the action to be strongly transitive is quite dramatic.
\end{abstract}

\maketitle

\section{Introduction}
\label{sec:intro}

Let~$D$ be an~$F$-division algebra of degree~$d>2$, i.e., $\dim_F(D)=d^2$. Let~$K|F$ be any splitting field of~$D$, and consider~$D$ as an~$F$-subalgebra of $M_d(K)$. Let $D^\times$ be the multiplicative group of~$D$, so $D^\times \le \GL_d(K)$. The group $\GL_d(K)$ acts strongly transitively on the corresponding spherical building. One wonders, however, whether the action of the subgroup~$D^\times$ is also strongly transitive, or even what we call \emph{weakly} transitive; see Definition~\ref{weak_trans_def}. The key lemma we use is Lemma~\ref{key}. For~$D^\times$ to act weakly transitively, some conjugate of~$D^\times$ in~$\GL_d(K)$ must intersect every coset in~$N/T$, with notation explained in Section~\ref{sec:bldgs}. As shown in Section~\ref{sec:div_algs}, there exist ``many" cosets for which this is impossible, and in fact the action fails ``dramatically" to be weakly transitive. These results hold as well for~$\SL_1(D)\le \SL_d(K)$.

In case~$F$ is a global field, we can choose~$K$ such that $\SL_d(K)$ acts on an \emph{affine} building~$\Delta_a$, strongly transitively with respect to the complete apartment system. Then $\SL_1(D)$ does act Weyl transitively on $\Delta_a$, but as we show does not act strongly transitively with respect to any apartment system. Group theoretically, a Weyl transitive group that is not strongly transitively with respect to any apartment system corresponds to a Bruhat decomposition not arising from a $BN$-pair. Whether or not~$\SL_1(D)$ can admit any non-trivial~$BN$-pair at all is another interesting question. Recent work of G.~Prasad~\cite{prasad12} shows that at least there are no~$BN$-pairs that are \emph{weakly split}.

Using $\SL_1(D)$ to find Weyl transitive but not strongly transitive actions stems from a suggestion of Tits, in Section~3.1 of his paper \cite{tits92}. In a paper by Abramenko and Brown \cite{abramenko07} the~$d=2$ case is shown to produce both examples and non-examples. Also, in \cite{zar_chev} examples of such actions were found using a different method, yielding results for buildings of arbitrary affine type. In the present work we show that for~$d>2$, $\SL_1(D)$ never acts strongly transitively on $\Delta_a$, and its failure to do so is quite extreme; see Section~\ref{sec:div_algs}.

If~$F$ is~$\Q$ or some other global field with certain properties, we prove some related results concerning cyclotomic subfields. Namely if~$D$ is a division algebra with center~$\Q$, or one of certain other global fields, and degree~$d>2$, then~$D$ cannot contain a primitive~$2d^{\text{th}}$ root of unity. Also, if~$d$ is odd,~$D$ cannot contain a primitive~$d^{\text{th}}$ root of unity. This further restricts the action of~$D^\times$ and $\SL_1(D)$ on the buildings; see Section~\ref{sec:div_algs_Q} for details.

Lastly we consider the case when~$D$ is a crossed product and~$N/T$ is constructed with respect not to an arbitrary apartment but to the fundamental apartment (see Section~\ref{sec:bldgs} for definitions and notation). In this scenario, we have a \emph{precise} description of which cosets in~$N/T$ can be represented in~$D^\times$. Thinking of $N/T=W\cong S_d$, the Weyl group elements representable in~$D^\times$ form a subgroup of order~$d$, isomorphic to a certain Galois group. Also, we see that the question of which cosets are representable in $\SL_1(D)$ is at least as difficult as determining the order of $[D]$ in the Brauer group. For~$F$ a global field we obtain that $\SL_1(D)$ can only represent the trivial coset and in some very restrictive cases one additional coset. In particular in the~$d=2$ case we recover the main result of \cite{abramenko07}. See Section~\ref{sec:fund_apt} for details.

It should be noted that there is a nice elegant proof due to A.~Rapinchuk, in some more generality, that $D^\times$ misses at least \emph{some} coset in~$N/T$ \cite{rapinchuk}. Also, arguments of K.~Tent~\cite{tent} establish this fact, and imply that $D^\times$ misses the vast majority of cosets in~$N/T$. In the present work examples of such cosets are explicitly constructed, and if $N/T$ is taken with respect to the fundamental apartment, then a precise characterization of such cosets is often possible.

\subsection*{Acknowledgments}

This paper is based on part of the author's Ph.D.~thesis work at the University of Virginia. He acknowledges the support of the department, and his advisor Peter Abramenko. He is also grateful to Andrei Rapinchuk and Katrin Tent for helpful conversations and for explaining simplified proofs of certain results.


\section{Buildings}\label{sec:bldgs}

Let $(W,S)$ be a Coxeter system,~$\Delta$ a building of type $(W,S)$ and~$G$ a group acting on~$\Delta$ via type-preserving automorphisms. Let $\C$ denote the set of chambers of~$\Delta$ and $\delta \colon \C\times\C \to W$ the Weyl distance function. In general, we denote apartment systems by $\A$ and the complete apartment system by~$\Abar$. For $C\in\C$ and $\A$ an apartment system, let $\A(C) \defeq \{\Sigma\in\A\mid C\in~\A\}$. For $\Sigma\in\Abar$ let $\C(\Sigma) \defeq \{C\in\C\mid C\in\Sigma\}$. We will not define or even describe any of these objects here; definitions and a wealth of details can be found throughout \cite{abramenko08}.

There are three transitivity properties that are of interest.

\begin{definition}\label{weak_trans_def}
We say the action of~$G$ on~$\Delta$ is \emph{weakly transitive} if there exists an apartment~$\Sigma\in\Abar$ such that $\Stab_G(\Sigma)$ acts transitively on $\C(\Sigma)$.
\end{definition}

\begin{definition}
We say a chamber transitive action of~$G$ on~$\Delta$ is \emph{strongly transitive with respect to} a~$G$-invariant apartment system $\A$ if it is transitive on $\A$ and there exists an apartment $\Sigma\in\A$ such that $\Stab_G(\Sigma)$ acts transitively on $\C(\Sigma)$, or equivalently if there exists~$C\in\C$ such that $\Stab_G(C)$ acts transitively on $\A(C)$. See \cite[Definition~6.1]{abramenko08}.
\end{definition}

\begin{definition}
We say a chamber transitive action of~$G$ on~$\Delta$ is \emph{Weyl transitive} if there exists $C\in\C$ such that $\Stab_G(C)$ acts transitively on the ``$w$-sphere" $\{D\in\C \mid \delta(C,D)=w\}$ for all $w\in W$. See \cite[Definition~6.10]{abramenko08}.
\end{definition}

One advantage of Weyl transitivity is that it makes no reference to apartments, and so does not depend on a choice of apartment system. We state here a proposition relating the three conditions.

\begin{proposition}\label{str_wk_wl}
If~$G$ acts strongly transitively with respect to $\A$ then it acts Weyl transitively and weakly transitively. If~$G$ acts Weyl transitively and weakly transitively then~$G$ acts strongly transitively with respect to some apartment system.
\end{proposition}

The second implication is stated for completeness, though we will not use it here. As such, we will not worry about its impreciseness regarding choice of apartment systems, though this is easy to fix. See \cite[Proposition~6.14]{abramenko08} for the proof of the proposition and a more precise version of the second implication. The important thing for us is, to show a group action is not strongly transitive with respect to \emph{any} apartment system, it suffices to show that there exists no $\Sigma\in\Abar$ satisfying the weak transitivity condition.

We are progressing toward our key Lemma~\ref{key}, which gives a group-theoretic condition for weak transitivity. We now suppose that~$G$ acts strongly transitively on~$\Delta$ with respect to~$\Abar$. For $\Sigma_0\in\Abar$ set $N \defeq \Stab_G(\Sigma_0)$, and let $T =\Fix_G(\Sigma_0) \defeq \{t\in N\mid tC=C$ for all chambers $C$ in $\Sigma_0\}$. Since~$G$ acts strongly transitively,~$N/T$ is isomorphic to the group of all type-preserving automorphisms of~$\Sigma_0$, i.e., to~$W$.

We now prove the key lemma regarding the action of subgroups~$H$ of~$G$.

\begin{lemma}\label{key}
Let~$H$ be a subgroup of~$G$. Then~$H$ acts weakly transitively on~$\Delta$ if and only if there exists~$g\in G$ such that $(gHg\I\cap N)T=N$.
\end{lemma}

\begin{proof}
Suppose~$H$ acts weakly transitively on~$\Delta$. Choose $\Sigma\in\Abar$ such that $\Stab_{H}(\Sigma)$ acts transitively on~$\C(\Sigma)$. Since~$G$ acts transitively on~$\Abar$ we can choose~$g\in G$ such that $g\Sigma=\Sigma_0$. Then $\Stab_{gHg\I}(\Sigma_0)=gHg\I\cap N$ acts transitively on~$\C(\Sigma_0)$, and so $(gHg\I\cap N)T=N$. Conversely, if $(gHg\I\cap N)T=~N$ then $gHg\I\cap N=\Stab_{gHg\I}(\Sigma_0)$ acts transitively on~$\C(\Sigma_0)$. But this means $\Stab_H(g\I\Sigma_0)$ acts transitively on $\C(g\I\Sigma_0)$, and so~$H$ acts weakly transitively.
\end{proof}

\begin{corollary}\label{not_str_trans}
Let~$H$ be a subgroup of~$G$. Suppose for all~$g\in G$ there exists $n\in N$ such that $gnTg\I\cap H=\emptyset$. Then the action of~$H$ on~$\Delta$ is not weakly transitive, and in particular is not strongly transitive with respect to any apartment system.
\end{corollary}
\begin{proof}
This is immediate from Proposition~\ref{str_wk_wl} and Lemma~\ref{key}.
\end{proof}

\section{Division algebras}\label{sec:div_algs}

We now specialize to the situation of interest. Let~$F$ be any field and~$D$ a division algebra of degree~$d>2$ with center~$F$. Consider the multiplicative group~$D^\times$ of~$D$. Let~$K$ be any splitting field of~$D$, so we can think of~$D$ as a subalgebra of $M_d(K)$, and~$D^\times$ as a subgroup of~$\GL_d(K)$. What we are really doing is identifying~$D$ with its image in $M_d(K)$ under some splitting representation $D\hookrightarrow M_d(K)$, but we will not run into any difficulties by ignoring the choice of representation and so will just treat~$D$ as a subalgebra of $M_d(K)$.

Now let~$\Delta$ be the spherical building associated to $\GL_d(K)$, with fundamental apartment~$\Sigma_0$. Then $N=\Stab_{\GL_d(K)}(\Sigma_0)$ is the subgroup of monomial matrices, $T=\Fix_{\GL_d(K)}(\Sigma_0)$ is the subgroup of diagonal matrices, and the Weyl group $W=N/T$ is naturally isomorphic to the symmetric group~$S_d$ \cite[Section~6.5]{abramenko08}.

We collect some facts about subfields of division algebras, all of which follow for example from \cite[Corollary~3.17]{farb93}.

\begin{proposition}\label{div_alg_facts}

For $E$ a subfield of~$D$, the following hold:

\noindent 1. The degree~$[E \colon F]$ divides~$d$.

\noindent 2. The degree~$[E \colon F]$ equals~$d$ if and only if $E$ is a maximal subfield of~$D$.

\noindent 3. If $E$ is a maximal subfield of~$D$, then~$D$ splits over $E$.
\end{proposition}

One further crucial fact is the following, which follows from \cite[Section~16.1]{pierce82}.

\begin{proposition}\label{char_poly_coeffs}
For $z\in D$, the characteristic polynomial $\chi_z(t)$ is independent of the choice of splitting field~$K$, and has coefficients in~$F$. In particular the trace of any $z\in D$ lies in~$F$.
\end{proposition}

We can now immediately show that~$D^\times$ does not act weakly transitively on~$\Delta$.

\begin{theorem}\label{trace_arg}
There exists a coset~$X$ in~$N/T$ such that for any $g\in\GL_d(K)$, we have $g D^\times g\I\cap X=\emptyset$.
\end{theorem}
\begin{proof}
Let $w\in S_d$ be any permutation that fixes precisely one element of $\{1,\dots,d\}$. Since $d>2$ such a~$w$ exists, and without loss of generality we may suppose that fixed point is~$d$. Let~$X$ denote the coset in~$N/T$ corresponding to~$w$, so~$X$ consists only of matrices of the form $A=\diag(B,c)$, where $B$ is a $(d-1)$-by-$(d-1)$ monomial matrix with zeros on its diagonal.

Now let $z\in D^\times$, $g\in\GL_d(K)$ and suppose $gzg\I\in X$. Say $gzg\I=A=\diag(B,c)$. Since~$z$ has trace $c$, by Proposition~\ref{char_poly_coeffs} we see that $c\in F$. Hence $c\in D$, and so also $z-c\in D$ (recall we have suppressed the embedding $D\hookrightarrow M_d(K)$). Thus, the matrix $g(z-c)g\I$ must either be zero or invertible. Since $gzg\I=A$ and $c=cI_d$ is central in $\GL_d(K)$, we have $g(z-c)g\I = A-cI_d = \diag(B-cI_{d-1},0)$. This is not invertible since the bottom row consists of zeros. Also, if $B-cI_{d-1}$ is zero then by construction of $B$, we have $c=0$. This contradicts the assumption that~$z$ is invertible.
\end{proof}

\begin{corollary}\label{trace_bldg}
$D^\times$ does not act weakly transitively on~$\Delta$.
\end{corollary}
\begin{proof}
This is immediate from Corollary~\ref{not_str_trans} and Theorem~\ref{trace_arg}.
\end{proof}

\begin{remark}
We could just as well replace $\GL$ with $\SL$ and all the preceding results would follow similarly. In this context we consider not~$D^\times$ but rather the norm-1 group $\SL_1(D)$.
\end{remark}

We now suppose~$K$ is complete with respect to a discrete valuation, and focus on the $\SL_1(D)$ situation. There is a standard \emph{affine} building on which $\SL_d(K)$ acts strongly transitively with respect to~$\Abar$. The spherical Weyl group~$W$ can be thought of in a natural way as a subgroup of the affine Weyl group $W_a$, and if conjugates of $\SL_1(D)$ fail to represent every element of~$W$ then they of course also fail to represent every element of $W_a$. We conclude the following:

\begin{corollary}\label{divalg_never_str_trans}
Let~$F$ be any field,~$D$ a central~$F$-division algebra of degree~$d>2$. Consider $\SL_1(D)\le\SL_d(K)$ for a splitting field~$K|F$. Let~$\Delta$ be the standard spherical building associated to $\SL_d(K)$, and $\Delta_a$ the standard affine building if~$K$ is complete with respect to a discrete valuation. Then the actions of $\SL_1(D)$ on~$\Delta$ and $\Delta_a$ are not weakly transitive, and are thus not strongly transitive with respect to any apartment system.\hfill\qed
\end{corollary}

\begin{remark}\label{2_case_and_weylT}
In these results it is important that~$d>2$. As seen in \cite{abramenko07}, if~$d=2$ the action \emph{can} in fact be weakly and even strongly transitive. Also, note that $\SL_1(D)$ could still act Weyl transitively on $\Delta_a$, provided that $\SL_1(D)$ is dense in $\SL_d(K)$. We thus have explicit examples of Weyl transitive, not-strongly transitive actions. Equivalently, we have examples of Tits subgroups that do not come from a $BN$-pair; see \cite[Chapter~6]{abramenko08} for the relevant definitions. Lastly, we mention again that there is a short proof of Corollary~\ref{divalg_never_str_trans} due to A.~Rapinchuk \cite{rapinchuk}, in fact in far more generality, and the result in the corollary also follows from an independent argument of K.~Tent \cite{tent}.
\end{remark}

Having shown that the actions fail to be weakly transitive, we now justify the claim that this failure is ``dramatic." Namely, we will exhibit a large collection of cosets~$X$ in~$N/T$ that have empty intersection with any conjugate of~$D^\times$. From now on we equate~$N/T$ with~$S_d$, and may refer to cosets by their cycle decomposition. This next proof is a direct generalization of the ``unique fixed point" situation from Theorem~\ref{trace_arg}.

\begin{theorem}\label{small_cycles}
Let~$X$ be any coset in~$N/T$ whose cycle decomposition features a unique cycle of minimum length that is not a~$d$-cycle, where a 1-cycle represents a fixed point. Let $z\in D^\times$ and $g\in\GL_d(K)$. Then $gzg\I\not\in X$.
\end{theorem}
\begin{proof}
Suppose $gzg\I=A\in X$ for some~$g$,~$z$. By adjusting~$g$ as necessary we may assume~$A$ is of the form $A=\diag(B_1,\dots,B_r,C)$ where
\[
 B_i=\begin{pmatrix}
  0 & _1b_i & 0 & \cdots & 0 \\
  0 & 0 & _2b_i & \cdots & 0 \\
  \vdots  & \vdots  & \ddots & \ddots & \vdots \\
  0 & 0 & 0 & \ddots & _{k_i-1}b_i \\
  _{k_i}b_i & 0 & 0 & \cdots & 0
 \end{pmatrix} \textnormal{ and  }
 C=\begin{pmatrix}
  0 & c_1 & 0 & \cdots & 0 \\
  0 & 0 & c_2 & \cdots & 0 \\
  \vdots  & \vdots  & \ddots & \ddots & \vdots \\
  0 & 0 & 0 & \ddots & c_{\ell-1} \\
  c_{\ell} & 0 & 0 & \cdots & 0
 \end{pmatrix}
\]
with $\ell<k_i$ for all $1\le i\le r$. Let $b_i \defeq {_1b_i}\cdots{_{k_i}b_i}$ for each~$i$ and $c \defeq c_1\cdots c_{\ell}$.

Consider the characteristic polynomial $\chi_A(t)=\chi_{B_1}(t)\cdots\chi_{B_r}(t)\chi_C(t)$. We know $\chi_A(t)$ equals $\chi_z(t)$, so by Proposition~\ref{char_poly_coeffs} its coefficients lie in~$F$. The constant term is $\pm b_1\cdots b_rc$, so we see that $b_1\cdots b_rc\in F$. The~$t^\ell$ term is also of interest. Since $\chi_{B_i}(t)=t^{k_i}-b_i$ and $\chi_C(t)=t^{\ell}-c$, and since $\ell<k_i$ for all~$i$, we see that the~$t^\ell$ term of $\chi_A(t)$ must be $\pm b_1\cdots b_rt^{\ell}$. Thus, $b_1\cdots b_r\in F$, and we conclude that $c\in F$.

In particular $z^{\ell}-c\in D$, so $A^{\ell}-cI_d$ is either invertible or is the zero matrix. Since $C^{\ell}=cI_{\ell}$, it is impossible that $A^{\ell}-cI_d$ can be invertible. But $\ell<k_i$ for all~$i$, so the $B_i^{\ell}-cI_{k_i}$ are all nonzero. We conclude that in fact $gzg\I\not\in X$ for any~$g$,~$z$.
\end{proof}

Theorem~\ref{small_cycles} discounts any coset featuring a ``unique smallest" cycle, and we can also discount cosets featuring a ``big" cycle, as the next theorem will show. First we need a simple lemma.

\begin{lemma}\label{min_poly_ncycle}
The matrix
\[
 A=\begin{pmatrix}
  0 & a_1 & 0 & \cdots & 0 \\
  0 & 0 & a_2 & \cdots & 0 \\
  \vdots  & \vdots  & \ddots & \ddots & \vdots \\
  0 & 0 & 0 & \ddots & a_{d-1} \\
  a_d & 0 & 0 & \cdots & 0
 \end{pmatrix}
\]
has minimal polynomial $t^d-a$ where $0\neq a \defeq a_1\cdots a_d$.
\end{lemma}
\begin{proof}
 $A$ satisfies this polynomial. Also, a quick calculation shows that $I_d, A,A^2,\dots,A^{d-1}$ are linearly independent, so the minimal polynomial of~$A$ cannot have degree less than~$d$. We conclude that $t^d-a$ is the minimal polynomial.
\end{proof}

\begin{theorem}\label{big_cycles}
Let~$X$ be any coset in~$N/T$ whose cycle decomposition features a~$k$-cycle, with $d/2<k<d$. Let $z\in D^\times$, $g\in\GL_d(K)$. Then $gzg\I\not\in X$.
\end{theorem}
\begin{proof}
Suppose $gzg\I=A\in X$ for some~$g$,~$z$. By adjusting~$g$ as necessary we may assume~$A$ is of the form $A=\diag(A',B)$ where
\[
 A'=\begin{pmatrix}
  0 & a_1 & 0 & \cdots & 0\\
  0 & 0 & a_2 & \cdots & 0\\
  \vdots  & \vdots  & \ddots & \ddots & \vdots\\
  0 & 0 & 0 & \ddots & a_{k-1}\\
  a_k & 0 & 0 & \cdots & 0
 \end{pmatrix}
\]
and $B$ is a $d-k$ by $d-k$ matrix. Set $a \defeq a_1\cdots a_k$ and $j \defeq d-k$, so $0<j<k$. The characteristic polynomial $\chi_A(t)$ equals $(t^k-a)\chi_B(t)$, and $\chi_B(t)$ has degree~$j$. As before, the coefficients of $\chi_A(t)$ lie in~$F$. In particular the coefficient on the $t^j$ term is in~$F$. But this term must be $-at^j$, since $k>j$ and $\chi_B(t)$ is monic of degree~$j$. Thus, $a\in F$. Also note that the minimal polynomial of~$z$ must have degree dividing~$d$, by Proposition~\ref{div_alg_facts}, but by Lemma~\ref{min_poly_ncycle} it also must have degree at least~$k$. Since $k>d/2$ we conclude that $\chi_A(t)$ is the minimal polynomial of~$z$.

Now, since $a\in F$, $z^k-a\in D$ and $g(z^k-a)g\I=A^k-aI_d$ is either zero or invertible. It cannot be invertible, so we conclude $z^k-a=0$. But $\chi_A(t)$ is the minimal polynomial of~$z$ and $k<d$ so this is impossible.
\end{proof}

We can generalize these two situations simultaneously with another criterion we call ``lonely cycles." Let~$\sigma$ be a permutation with cycle decomposition $\sigma=\sigma_1\cdots\sigma_m$ for each~$\sigma_i$ a~$k_i$-cycle (as usual we account for fixed points with ``1-cycles"). We will call~$\sigma_i$ \emph{lonely} if~$k_i$ satisfies the following two conditions:
\begin{enumerate}
 \item For any $\epsilon_1,\dots,\epsilon_r\in\{0,1\}$, if $\displaystyle k_i=\sum_{j=1}^r\epsilon_jk_j$ then $\epsilon_j=0$ for all $j\neq i$.
 \item If~$k_i$ is maximal among all~$k_j$ then~$k_i$ does not divide~$d$.
\end{enumerate}

For example the $3$-cycle $(1~2~3)$ is lonely in the permutation given by $(1~2~3)(4~5)(6~7)$ since~$3$ cannot be written as a sum involving~$2$ and~$2$, and~$3$ does not divide~$7$, but the~$2$-cycle~$(4~5)$ is \emph{not} lonely since~$2$ \emph{can} be written as a sum involving~$3$ and~$2$, namely~$2=2$. Note that the second condition in particular ensures that~$d$-cycles themselves are not lonely. However, the~$d$-cycles are still an interesting case, which we will consider later. For now we claim that any permutation featuring a lonely cycle cannot be represented by~$D^\times$.

\begin{theorem}\label{lonely_cycles}
Let~$X$ be any coset in~$N/T$ whose cycle decomposition features a lonely~$k$-cycle. Let $z\in D^\times$, $g\in\GL_d(K)$. Then $gzg\I\not\in X$.
\end{theorem}
\begin{proof}
Suppose $gzg\I=A\in X$ for some~$g$,~$z$. Let the cycle type of~$X$ be $j_1,\dots,j_r,k$, so no collection of distinct~$j_i$ can sum to~$k$. By adjusting~$g$ as necessary we may assume~$A$ is of the form $A=\diag(B_1,\dots,B_r,C)$ where
\[
 B_i=\begin{pmatrix}
  0 & _ib_1 & 0 & \cdots & 0\\
  0 & 0 & _ib_2 & \cdots & 0\\
  \vdots  & \vdots  & \ddots & \ddots & \vdots\\
  0 & 0 & 0 & \ddots & _ib_{j_i-1}\\
  _ib_{j_i} & 0 & 0 & \cdots & 0
 \end{pmatrix}
\]
and \[
 C=\begin{pmatrix}
  0 & c_1 & 0 & \cdots & 0\\
  0 & 0 & c_2 & \cdots & 0\\
  \vdots  & \vdots  & \ddots & \ddots & \vdots\\
  0 & 0 & 0 & \ddots & c_{k-1}\\
  c_k & 0 & 0 & \cdots & 0
 \end{pmatrix}.
\]
Set $b_i \defeq {_ib_1}\cdots{_ib_{j_i}}$ for each~$i$ and set $c \defeq c_1\cdots c_k$. The characteristic polynomial $\chi_A(t)$ equals $\chi_{B_1}(t)\cdots \chi_{B_r}(t)\chi_C(t)$. The coefficients of $\chi_A(t)$ lie in~$F$, so in particular the coefficient on the~$t^k$ term is in~$F$. But this term must be $\pm(b_1\cdots b_r)t^k$, since $\chi_{B_i}(t)=t^{j_i}-b_i$, $\chi_C(t)=t^k-c$, and no collection of~$j_i$ can sum to~$k$. Thus, $b_1\cdots b_r\in F$. The constant term of $\chi_A(t)$ is $\pm b_1\cdots b_rc$, so we also know that $c\in F$.

In particular, $z^k-c\in D$ and $g(z^k-c)g\I=A^k-cI_d$ is either zero or invertible. It cannot be invertible, so $z^k-c=0$. This implies that each~$j_i$ divides~$k$, and so~$k$ is maximal among $j_1,\dots,j_r,k$. Then by the second criterion in the definition of lonely cycles,~$k$ does not divide~$d$. However, by the proof of Lemma~\ref{min_poly_ncycle} the minimal polynomial of~$z$ cannot have degree smaller than~$k$, so in fact $z^k-c$ is the minimal polynomial of~$z$. Since the degree of the subfield $F(z)$ in~$D$ must divide~$d$ by Proposition~\ref{div_alg_facts}, this is a contradiction.
\end{proof}

As an example, consider the permutation $\sigma=(1~2~3)(4~5)(6~7)$ from earlier. Here $(1~2~3)$ is lonely, so no conjugate of~$D^\times$ can represent~$\sigma$. Note that~$\sigma$ does not feature a ``big" cycle, nor a ``unique smallest" cycle, so we have really gained ground by considering lonely cycles. Also note that any big or unique smallest cycle is clearly lonely, so this is really a generalization.

All these results of course still hold for $\SL_1(D)$. We now inspect this case further and show that there may be even more cosets~$X$ in~$N/T$ that fail to be represented by any conjugate of $\SL_1(D)$, namely the~$d$-cycles.

\begin{theorem}\label{min_poly_thm}
Suppose~$d$ is not a power of 2. For any $z\in\SL_1(D)$, $g\in\GL_d(K)$, $gzg\I$ cannot be of the form
\[
 A=\begin{pmatrix}
  0 & a_1 & 0 & \cdots & 0 \\
  0 & 0 & a_2 & \cdots & 0 \\
  \vdots  & \vdots  & \ddots & \ddots & \vdots \\
  0 & 0 & 0 & \ddots & a_{d-1} \\
  a_d & 0 & 0 & \cdots & 0
 \end{pmatrix}
\]
\end{theorem}
\begin{proof}
Suppose $gzg\I=A$. By Lemma~\ref{min_poly_ncycle},~$z$ has minimal polynomial $t^d-a$ where $a=a_1\cdots a_d$. Since~$z$ has norm~$1$, $a=(-1)^{d+1}$ so~$z$ has minimal polynomial $t^d+(-1)^d$. Since~$D$ is a division algebra and $z\in D$, $t^d+(-1)^d$ is irreducible.

This is impossible if~$d$ is odd, since~$1$ is a root of $t^d-1$, so assume~$d$ is even. Let $d=2^em$ for odd~$m$. Then $t^{2^e}+1$ divides $t^d+1$, so by irreducibility we know that~$d$ must be a power of~$2$. But we discounted this possibility in the hypothesis.
\end{proof}

In general, we can justify calling the failure to act weakly transitively ``dramatic." For sufficiently large~$d$, it can be shown that over~$70\%$ of the permutations in~$S_d$ feature a ``big" cycle, and thus cannot be represented by any conjugate of~$D^\times$. We note that accounting for the~$d$-cycles and the unique smallest cycles does not improve this estimate, and at present it is not clear whether this estimate would change by accounting for permutations featuring a lonely cycle.

\section{Division algebras over~$\Q$}\label{sec:div_algs_Q}
In this section we prove some further results in the case that~$F$ is either~$\Q$ or another global field with certain properties. Specifically we will use a very different method to prove a version of Theorem~\ref{min_poly_thm} even when~$d$ is a (proper) power of~$2$. We also achieve an interesting result regarding cyclotomic subfields of division algebras.

We collect some facts about central simple algebras over~$\Q$ and $\Q_p$ for prime~$p$. For details on the Brauer group of local and global fields, and the Hasse invariant maps, see \cite[Section~31]{lorenz08}, \cite[Chapter~12]{serre79}. One thing to note is that the invariant maps take values in~$\Q/\Z$, but we will simply write their outputs as fractions in $[0,1)$.

\begin{proposition}\label{loc_glob}
Let~$A$ be a central simple algebra over~$\Q$. Let $A_p \defeq A\otimes_{\Q}\Q_p$ for each~$p$ prime or infinity, where $A_{\infty} \defeq A\otimes_{\Q}\R$. The following properties hold:

\noindent 1. For all but finitely many~$p$, $\inv_p(A)=0$.

\noindent 2. The sum of all the $\inv_p(A)$ is zero.

\noindent 3. The Schur index of~$A$ equals the \emph{exponent} of~$A$, i.e., the least common multiple \indent of the (reduced) denominators of the non-zero $\inv_p(A)$.

\noindent 4. If~$A_p$ splits over a finite extension $F|\Q_p$, then the index of~$A_p$ divides $[F \colon \Q_p]$.
\end{proposition}
\begin{proof}
Parts 1 and 2 are in the definition of $\Br(\Q)$ \cite[Chapter~12]{serre79}. Part 3 follows from the well-known Albert-Brauer-Hasse-Noether theorem. Part 4 follows from \cite[Section~13.3,~Corollary~1]{serre79}.
\end{proof}

There is another result that will be useful in the case when the degree of~$D$ is a power of a prime.

\begin{lemma}\label{pm_pow}
Suppose~$d>2$ is a power of a prime, $d=s^m$. Then there exists an odd prime~$p$ such that~$D_p$ is a division algebra.
\end{lemma}
\begin{proof}
Since~$D$ has index $d=s^m$, by Proposition~\ref{loc_glob} it also has exponent~$s^m$. Thus we know that some $\inv_p(D)$ has reduced denominator $s^m$. Note that $\inv_\infty(\Br(\R))=\{0,1/2\}$, so~$p$ must be prime. Also, since the invariants sum to zero, we know that there are actually \emph{at least two}~$p_1$,~$p_2$ such that $\inv_{p_1}(D)$ and $\inv_{p_2}(D)$ have denominator~$s^m$. Thus, we can choose an \emph{odd} prime~$p$ such that $\inv_p(D)$ has denominator~$s^m$. But this means that~$D_p$ has index~$d$ and so is a division algebra.
\end{proof}

As usual, for a splitting field $K|\Q$ of~$D$ we think of $\SL_1(D)$ as a subgroup of $\SL_d(K)$. Let~$N$,~$T$, and $W\cong S_d$ be associated to $\SL_d(K)$ as before. We want to show that $\SL_1(D)$ cannot represent the~$d$-cycles in~$W$, even if~$d$ is a power of 2.

Let~$A$ be the matrix
\[
 A=\begin{pmatrix}
  0 & a_1 & 0 & \cdots & 0 \\
  0 & 0 & a_2 & \cdots & 0 \\
  \vdots  & \vdots  & \ddots & \ddots & \vdots \\
  0 & 0 & 0 & \ddots & a_{d-1} \\
  a_d & 0 & 0 & \cdots & 0
 \end{pmatrix}
\]
\noindent in $\SL_d(K)$. Since $\det{A}=1$, we calculate that $\prod a_i=(-1)^{d-1}$. But we also see that $A^d=\prod a_iI_d$. Thus,~$A$ has multiplicative order either~$d$ or $2d$, for~$d$ odd or even respectively. Moreover, up to conjugation any representative of a~$d$-cycle in $N/T\cong S_d$ is of this form. To show that no conjugate of an element of $\SL_1(D)$ can represent a~$d$-cycle, it therefore suffices to prove the following:

\begin{theorem}\label{thm}
If~$d$ is odd then~$D^\times$ contains no elements of order~$d$, and if~$d$ is even then~$D^\times$ contains no elements of order~$2d$.
\end{theorem}

\begin{proof}
We first cover the case when $d=2^em$ for odd $m>1$, i.e., the case when~$d$ is not a power of~$2$. Suppose~$D$ contains a primitive $r^{\text{th}}$ root of unity, $\zeta_r$. Then~$D$ features the subfield~$L \defeq \Q(\zeta_r)$ of degree $\varphi(r)$, so by Proposition~\ref{div_alg_facts} $\varphi(r)$ must divide~$d$.

Now, if $r=d$ and~$d$ is odd, then $\varphi(r)$ cannot divide~$d$, since $\varphi(r)$ is even for $r>2$. Also, if~$r=2d$ and~$d$ is even, then $\varphi(r)=\varphi(2^{e+1}m)=2^e\varphi(m)$, and since $m\ge 3$, in this case~$2^{e+1}|\varphi(r)$ and so $\varphi(r)$ cannot divide~$d$. The remaining case is when $d=2^e$. Then~$\varphi(2d)=~d$, so $L=\Q(\zeta_{2d})$ is a maximal subfield of~$D$ and~$D$ splits over~$L$, again by Proposition~\ref{div_alg_facts}.

For prime~$p$, let $L_p \defeq \Q_p(\zeta_{2d})$. Since~$D$ splits over~$L$, also~$D_p$ splits over~$L_p$, for any~$p$. By Proposition~\ref{loc_glob}, for each~$p$ the index of~$D_p$ divides the degree of~$L_p$. We claim that for any odd prime~$p$, this degree satisfies $[L_p:\Q_p]<d$. This will contradict Lemma~\ref{pm_pow}, proving the theorem. By \cite[Section~4.4,~Corollary~1]{serre79}, $[L_p:\Q_p]$ equals the order of~$p$ in $(\Z/2^{e+1}\Z)^\times$. Since~$e>1$, this group is isomorphic to $\Z/2\Z\times\Z/2^{e-1}\Z$, and so~$p$ has order strictly less than~$2^e$. Thus, $[L_p:\Q_p]<d$.
\end{proof}

This is an interesting result in its own right, and also proves that $\SL_1(D)$ does not represent any~$d$-cycles in $N/T\cong S_d$, even in the case that~$d$ is a proper power of 2. There was nothing special about~$F=\Q$, except for the fact that $[\Q(\zeta_r):\Q]=\varphi(r)$. It is easy to see that these results hold for global fields~$F$ other than~$\Q$, provided the relevant cyclotomic extensions have the same degree as in the~$\Q$ case.

\section{Action on the fundamental apartment}\label{sec:fund_apt}

As we have seen, the action of~$D^\times$ on the relevant building is far from weakly transitive. We now analyze a related question, namely, we know that~$D^\times$ does not represent the whole Weyl group, but how much exactly \emph{does} it represent? Thinking of the Weyl group as the stabilizer modulo the fixer of an \emph{arbitrary} apartment, this problem seems difficult. If we consider the \emph{fundamental} apartment however, we can achieve a precise description, at least for ``most" division algebras~$D$.

Let~$D$ be a division algebra of degree~$d$, and suppose~$D$ has a maximal subfield~$K$ that is Galois over the center~$F$ of~$D$. (Division algebras lacking this property exist, but are difficult to construct; the first examples were discovered only in the 1970's, by Amitsur \cite{amitsur72}. Also for certain~$F$, e.g.,~$F$ global, all~$F$-division algebras have this property, and so this restriction is not very severe.) Let $\Gamma=\Gal(K|F)$. Using the notion of a crossed product, we can construct a (right)~$K$-basis $\{x_\sigma\}_{\sigma \in \Gamma}$ for~$D$, with multiplication given by $x_\sigma x_{\tau}=x_{\sigma\tau}a_{\sigma,\tau}$ for some $a_{\sigma,\tau} \in K^\times$ and $bx_\sigma = x_\sigma \sigma(b)$ for $b\in K$; see \cite[Theorem~30.1.1]{lorenz08} for details. The action of~$D^\times$ by left multiplication on the~$d$-dimensional right~$K$-vector space~$D$ is~$K$-linear, and so induces an injective homomorphism $D^\times \hookrightarrow \GL_d(K)$. We will suppress this map and just think of~$D^\times$ as a subgroup of $\GL_d(K)$.

Let~$\Delta$ be the standard spherical building for $\GL_d(K)$. Let~$\Sigma_0$ denote the fundamental apartment, with stabilizer $N$ the group of monomial matrices and fixer $T$ the group of diagonal matrices. Note that the Weyl group $W=N/T$ is isomorphic to~$S_d$. Consider the action of the subgroup~$D^\times$ on~$\Delta$. We claim that we can completely describe the subgroup $W_{D^\times} \defeq \Stab_{D^\times}(\Sigma_0)/\Fix_{D^\times}(\Sigma_0)$ of~$W$. Define $\phi \colon \Gamma \to W_{D^\times}$ to be $\sigma\mapsto x_\sigma\Fix_{D^\times}(\Sigma_0)$. Since $a_{\sigma,\tau}\in~K^\times$ fixes~$\Sigma_0$,~$\phi$ is a homomorphism. Also, if~$x_\sigma$ fixes~$\Sigma_0$ then in particular $x_\sigma x_\sigma$ lies in the~$K$-span of~$x_\sigma$, implying that $\sigma^2=\sigma$, so~$\sigma=1$. Hence~$\phi$ is injective. Lastly, since our choice of standard basis is unique up to~$K$-span \cite{lorenz08},~$\phi$ is a canonical map.

\begin{proposition}\label{div_alg_weyl_gp}
The map~$\phi$ defined above is surjective, and so is a canonical isomorphism.
\end{proposition}
\begin{proof}
Suppose that $z\in D^\times$ is a monomial matrix. Then for any basis element $x_\tau$, $zx_\tau$ is again in the~$K$-span of some basis element. Say $\displaystyle z=\sum_{\sigma \in \Gamma}x_\sigma b_\sigma$, so $\displaystyle zx_{\tau}=\sum_{\sigma \in \Gamma}x_{\sigma\tau} a_{\sigma,\tau} \tau(b_\sigma)$. For this to lie in the span of a single basis element, we must have that all but one of the~$b_\sigma$ are in fact zero. Thus~$z$ is of the form $x_\sigma b$ for some $b\in K^\times$, $\sigma\in\Gamma$. We conclude that the stabilizer of~$\Sigma_0$ in~$D^\times$ is made up precisely of elements of this form. Of course $x_\sigma b\Fix_{D^\times}(\Sigma_0) = x_\sigma \Fix_{D^\times}(\Sigma_0)$, and so indeed~$\phi$ is surjective.
\end{proof}

This yields a precise description of the action of~$D^\times$ on~$\Sigma_0$, given by the subgroup $W_{D^\times} \cong \Gamma$ of~$W$. In fact the map $\phi \colon \Gamma \hookrightarrow W_{D^\times}$ is explicit. Since $W=S_d$ and $|\Gamma|=d$ we can think of~$W$ as the symmetric group on the set~$\Gamma$. Then~$\phi$ is just induced by the left multiplication of~$\Gamma$ on itself. If we fix $\sigma \in \Gamma$ with order~$\ell$ and choose $\tau_1,\dots,\tau_r$ such that $\displaystyle \Gamma = \bigcup_{i=1}^r\langle\sigma\rangle\tau_i$, where~$r=d/\ell$, then $\phi(\sigma)$ has cycle decomposition $(\tau_1~\sigma\tau_1~\cdots~\sigma^{\ell-1}\tau_1)\cdots(\tau_r~\sigma\tau_r~\cdots~\sigma^{\ell-1}\tau_r)$, which in particular consists only of~$\ell$-cycles. This verifies a suggestion of M.~Kassabov that~$D^\times$ should only be able to represent $w\in W$ if the cycle decomposition of~$w$ features cycles all of the same length, at least for the case of the fundamental apartment.

As an aside, we note that the building and fundamental apartment depend on the choice of~$K$, and so the fact that this description depends on~$\Gamma$ is not surprising. That is, if~$D$ contains some other Galois maximal subfield~$K'$ with Galois group $\Gamma'\not\cong\Gamma$, then the resulting building and fundamental apartment are different, and so $W_{D^\times}$ will be different. It would thus be most precise to use the notation $W_{D^\times,K,\Sigma_0}$ but for brevity we will just write $W_{D^\times}$.

Thanks to the description of $W_{D^\times}$ we can also say something about the action of $\SL_1(D)$ on~$\Sigma_0$. Namely, $\SL_1(D)$ represents $\phi(\sigma)$ in $W_{D^\times}$ if and only if there exists $b\in K$ such that $x_\sigma b$ has reduced norm~$1$. This in turn will happen if and only if the reduced norm of~$x_\sigma$ is already a Galois norm of something in~$K$. In general the question of whether the reduced norm of~$x_\sigma$ is a Galois norm is a difficult one, and so the precise determination of $W_{\SL_1(D)} \defeq \Stab_{\SL_1(D)}(\Sigma_0)/\Fix_{\SL_1(D)}(\Sigma_0)$ is difficult. We now specialize to the case of cyclic algebras and see that already this question is at least as difficult as a well-known problem that remains unsolved in many cases.

Let~$D$ be a cyclic algebra $D=(K/F,\sigma,a)$ with $\Gal(K|F)=\langle\sigma\rangle$ and the above standard~$K$-basis now given by $x_{\sigma^i}=x^i$ for $0\le i\le d-1$, where $x \defeq x_\sigma$ and $x^d=a\in F^\times$; for details about cyclic algebras see \cite[Chapter~31]{lorenz08}. In particular for $0\le i,j\le {d-1}$, $a_{\sigma^i,\sigma^j}$ is 1 if $i+j<d$ and is~$a$ otherwise. This allows us to explicitly calculate the reduced norm of the basis elements, namely $x^i$ has reduced norm $(-1)^{i(d-1)}a^i$ for each~$i$. Thus for each~$i$, $\phi(\sigma^i)$ is represented by $\SL_1(D)$ if and only if $(-1)^{i(d-1)}a^i$ is a Galois norm of something in $K^\times$.


Consider the ($\sigma$-dependent) isomorphism $\Br(K|F) \to F^\times/N(K^\times)$ under which $[D]\mapsto aN(K^\times)$, described in \cite[Theorem~30.4.4]{lorenz08}. Thanks to this isomorphism we see that for any~$i$,~$a^i$ is a norm if and only if $e(D)|i$, where $e(D)$ is the order of~$[D]$ in $\Br(K|F)$. This shows that determining $W_{\SL_1(D)}$ is essentially equivalent to determining~$e(D)$, and in general it is an open question to determine~$e(D)$ for arbitrary~$D$. We now consider the global case, where we can say much more, and in particular can precisely calculate $W_{\SL_1(D)}$.

Suppose~$F$ is global. Then~$D$ is automatically cyclic, say $D=(K/F,\sigma,a)$, and $e(D)=d$. Both of these facts are results of the well-known Albert-Brauer-Hasse-Noether theorem \cite{roquette05}, \cite[Theorem~10.7(a)]{saltman99}. If $\SL_1(D)$ represents $\phi(\sigma^i)$, then the element $(-1)^{i(d-1)}a^i$ is a Galois norm, and so its square $a^{2i}$ is as well. This tells us that~$i$ can only possibly be~$d$, or $d/2$ if~$d$ is even. The $i=d$ case corresponds to the trivial Weyl group element, which is of course in $W_{\SL_1(D)}$. Suppose now that $i=d/2$. Since $a^i$ is not a Galois norm, $(-1)^{i(d-1)}$ must be $-1$, i.e.,~$i$ is odd and~$d$ is congruent to 2 mod 4. We now have a complete characterization of $W_{\SL_1(D)}$ for the case when~$F$ is global, described in the following:

\begin{theorem}\label{norm_1_weyl_gp}
Let~$F$ be a global field. Let $D=(K/F,\sigma,a)$ be as above, with degree~$d$. If~$d$ is not congruent to 2 mod 4, or if it is but $-a^{d/2}$ is not in $N_{K|F}(K^\times)$, then $W_{\SL_1(D)}=\{1\}$. If~$d$ is congruent to 2 mod 4 and $-a^{d/2}$ \emph{is} in $N_{K|F}(K^\times)$, then $W_{\SL_1(D)}=\{1,\phi(\sigma^{d/2})\}$.\qed
\end{theorem}

We observe that when~$d=2$ this proposition says that $\Stab_{\SL_1(D)}(\Sigma_0)$ acts transitively on~$\C(\Sigma_0)$ if and only if $-a\in N_{K|F}(K^\times)$. In \cite{abramenko07}, it is shown that if~$d=2$ and $F=\Q$ then $\SL_1(D)$ acts weakly transitively if and only if $-1\in D^2$. The conditions $-a\in N_{K|F}(K^\times)$ and $-1\in D^2$ would thus seem to be related. Indeed if $-a\in N_{K|F}(K^\times)$, say $-a=b\sigma(b)$ for~$b\in K$, then $(xb\I)^2=x^2\sigma(b\I)b\I=a(-a)\I=-1$ so $-1$ is a square. The converse need not be true, since $\SL_1(D)$ can act weakly transitively without doing so via~$\Sigma_0$. For instance if $K=\Q(i)$ and $a=-3$ then of course $-1\in D^2$ but $-a=3$ is not a Galois norm. In this case then $\SL_1(D)$ acts weakly transitively on~$\Delta$ via some apartment other than~$\Sigma_0$.

We conclude with a few words regarding affine buildings. Let~$F$ be a global field and $D=(K/F,\sigma,a)$ as above, with degree~$d$. For technical reasons we assume the characteristic of~$F$ does not divide~$d$. Suppose~$K$ is embedded in $F_{\nu}$ for some non-archimedean valuation~$\nu$ of~$F$. Let $F_{\nu}$ have valuation ring $R$ and residue field~$k$. Let $\Delta_a$ be the affine building on which~$\SL_d(F_{\nu})$ acts strongly transitively with respect to~$\Abar$, and let $X_0$ denote the fundamental affine apartment.

Let~$W$, $N$, and $T$ be the standard spherical data for $\SL_d(F_{\nu})$, and let $W_a$, $T_a$ be the standard affine data. Then $W_a=N/T_a$ and $W=N/T$, so if $Q \defeq T/T_a$ we have the short exact sequence
$$1 \to Q \to W_a \to W \to 1.$$
In fact the inclusion $W\hookrightarrow W_a$ provides a splitting, and we get $W_a=W\ltimes Q$ \cite[Section~6.9.2]{abramenko08}. We call $Q$ the group of \emph{translations}.


It turns out that $\SL_1(D)$ represents every element of $Q$, i.e., $Q$ is contained in $(W_a)_{\SL_1(D)} \defeq \Stab_{\SL_1(D)}(X_0)/\Fix_{\SL_1(D)}(X_0)$; this follows for similar reasons as in the proof of Lemma~2.2 in~\cite{abramenko07}. This implies that $(W_a)_{\SL_1(D)}=W_{\SL_1(D)}\ltimes Q$. As we have seen $W_{\SL_1(D)}$ is usually trivial, and if not then it has order 2. Thus despite $\SL_1(D)$ acting Weyl transitively on $\Delta_a$, it acts on $X_0$ only by translations in most cases, and only by a product of translations with a single transposition in all other cases.


\bibliographystyle{alpha}
\bibliography{bibdata}


\end{document}